\documentclass[12pt]{amsart}

\usepackage{amsthm, amsmath, amssymb, amsrefs}

\newcommand{\T}{\mathbb{T}} 
\newcommand{\C}{\mathbb{C}} 
\newcommand{\D}{\mathbb{D}} 
\newcommand{\cd}{\overline{\D}} 

\numberwithin{equation}{section}

\newtheorem{theorem}{Theorem}[section]

\newtheorem{lemma}[theorem]{Lemma}

\theoremstyle{definition}
\newtheorem{definition}[theorem]{Definition}

\title{The von Neumann inequality for $3\times 3$ matrices}

\author{Greg Knese}

\address{Washington University in St. Louis\\ Department of
  Mathematics\\ St. Louis, Missouri 63130}

\email{geknese@math.wustl.edu}

\date{\today}

\thanks{Partially supported by NSF grant DMS-1363239}


\begin{document}
\maketitle

\section{Introduction}

The purpose of this note is to explain how recent results of
Kosi\'nski \cite{LK} provide a proof of the von Neumann inequality for
$d$-tuples of $3\times 3$ commuting contractive matrices.
 
\begin{definition}
A $d$-tuple of pairwise commuting contractive matrices or
operators $T = (T_1,\dots, T_d)$ satisfies the
\emph{von Neumann inequality} if for every $p \in \C[z_1,\dots,z_d]$
\[
\|p(T) \| \leq \sup_{z \in \T^d} |p(z)|.
\]
\end{definition}

Recall that contractive means $\|T_j\| \leq 1$ and $\T^d$ is the unit
$d$-torus in $\C^d$.

The von Neumann inequality holds for a single contractive
operator---von Neumann's original result \cite{vN}---as well as for a
pair of commuting contractions---a result of And\^{o} \cite{ando}.
For $d>2$ there are known examples of $d$-tuples of commuting
contractions for which the von Neumann inequality fails.  Varopoulos
\cite{varo} proved the existence of counterexamples with a
probabalistic argument and later Kaijser and Varopoulos (see addendum
to \cite{varo}) as well as Crabb and Davie \cite{Crabb} found explicit
counterexamples.  These counterexamples were all given by finite
matrices.  It turns out that the von Neumann inequality holds for
$d$-tuples of $2\times 2$ commuting contractive matrices; this result
is essentially equivalent to the Schwarz lemma on the polydisc.  On
the other hand, Holbrook \cite{Holbrook} has found a 3-tuple of
$4\times 4$ matrices which fail the von Neumann inequality.

A great deal of effort has been expended to put in the last piece of
this particular puzzle: Does the von Neumann inequality hold for
$d$-tuples of $3\times 3$ commuting contractive matrices?  For
evidence of interest in this question see \cite{ChoiDavidson,
  Holbrook, Lotto}.

Recent work of Kosi\'nski proves that the answer is yes.

\begin{theorem} \label{mainthm} The von Neumann inequality holds for
  $d$-tuples of $3\times 3$ commuting contractive matrices.
\end{theorem}

This is also interesting in light of the fact that there exists a
$4$-tuple of $3\times 3$ commuting contractions that do not
\emph{dilate} to commuting unitary operators \cite{ChoiDavidson}.  
Possessing a unitary dilation is a much stronger property than
satisfying a von Neumann inequality; in fact, it means that a von
Neumann inquality holds for matrix-valued polynomials of all matrix
sizes.  Nevertheless, all of the von Neumann inequalities cited above
have dilation counterparts: the Sz.-Nagy dilation theorem vis-\`a-vis
von Neumann's inquality, And\^{o}'s dilation theorem (And\^{o}'s
actual result in \cite{ando}), $d$-tuples of $2\times 2$ commuting
contractive matrices always dilate \cite{Drury, Holbrook2}.  The
situation of a von Neumann inequality without a dilation theorem is
not uncommon though, and it is closely related to the distinction
between spectral sets and complete spectral sets. See \cite{DM} for
more information.

In the next two sections, we point out some known reductions and then
explain how Kosi\'nski's work proves Theorem \ref{mainthm}.  In the
third and final section, we explain some other operator theory related
consequences of \cite{LK}.  Namely, solvable three-point Pick
interpolation problems on the polydisc can always be solved with a
rational inner function in the Schur-Agler class.

\section{Reductions}
Let $T=(T_1,\dots, T_d)$ be a $d$-tuple of $3\times 3$ commuting
contractive matrices.  To start with, we can assume the matrices are
strict contractions ($\|T_j\| < 1$, $j=1,\dots, d$) and then the theorem
will follow by continuity.  Next, a $d$-tuple of $3\times 3$ commuting
matrices can always be perturbed to a simultaneously diagonalizable
$d$-tuple of commuting matrices.  The reference \cite{HO} points out
several places where this is proven; see \cite{Gerstenhaber, Guralnick,
  Lotto}.

After adjusting our operators to be nicer, we will replace polynomials
with functions that are less nice.  If $T$ is now a $d$-tuple of
commuting simultaneously diagonalizable strictly contractive $3\times
3$ matrices, then it suffices to prove
\[
\|f(T)\| \leq 1
\]
for all $f:\D^d \to \D$ holomorphic on the unit $d$-dimensional
polydisc.  This follows from the fact that such holomorphic functions
can be approximated locally uniformly on $\D^d$ by polynomials $p \in
\C[z_1,\dots, z_d]$ with supremum norm at most $1$ on $\D^d$ (or
$\T^d$ by the maximum principle).  See Rudin \cite{Rudin}.  

Using bounded holomorphic functions makes it possible to apply
M\"obius transformations to the matrices $T_1,\dots, T_d$ in order to
force one of the joint eigenvalues of $T$ to be $0 \in \C^d$, while
still maintaining all other properties of $T$.  We can also apply a
M\"obius transformation to $f:\D^d \to \D$ and assume $f(0)=0$.

Let $0,z,w \in \D^d$ be the joint eigenvalues of $T$ with
corresponding eigenvectors $e, u, v\in \C^3$.  Then, $f(T)$ is the
$3\times 3$ matrix with eigenvalues $0, \sigma=f(z), \tau=f(w) \in \D$
and eigenvectors $e,u,v$.  It now becomes of interest to understand
all holomorphic functions $g:\D^d \to \D$ which solve the following
interpolation problem:
\begin{equation} \label{interp}
\begin{aligned}
0 &\mapsto 0 \\
z & \mapsto \sigma \\
w & \mapsto \tau
\end{aligned}
\end{equation}

Theorem \ref{mainthm} will follow from the next result which is
explained in the next section.

\begin{theorem}[Kosi\'nski] \label{interpthm} If the interpolation
  problem \eqref{interp} can be solved with $g:\D^d\to \D$
  holomorphic, then there exist holomorphic $F_1,F_2:\D^2 \to \D$ such
  that \eqref{interp} can be solved with a function of the form
\begin{equation} \label{form}
F(z) = F_1(F_2(z_1,z_2),z_3)
\end{equation}
after possibly permuting the variables.
\end{theorem}

Indeed, by And\^{o}'s inequality $S = F_2(T_1,T_2)$ is a contraction
commuting with $T_3$ and therefore $F(T) = F_1(S,T_3)$ is a
contraction equal to $f(T)$ as above.  This proves Theorem
\ref{mainthm} given Theorem \ref{interpthm}.

\section{The three point Pick problem on the polydisc}

Theorem \ref{interpthm} follows from work in \cite{LK} after making
some further reductions to put us in the most interesting situation
(that of \emph{extremal} and \emph{non-degenerate} interpolation
problems). 

First, it is useful to perturb the nodes $0, z=(z_1,\dots, z_d),
w=(w_1,\dots,w_d)$ into a more generic position.  Let $\rho(a,b) =
\left|\frac{a-b}{1-\bar{a}b}\right|$ be the pseudo-hyperbolic distance
on the unit disk.  Perturb $z,w$ so that all of the quantities below
are distinct, yet $T$ is still strictly contractive:
\[
|z_1|,\dots, |z_d|, |w_1|,\dots, |w_d|, \rho(z_1,w_1),\dots,
\rho(z_d,w_d).
\]
We will see momentarily why this is useful.  Functions as in
\eqref{form} form a normal family so we can approximate the less
generic interpolation problems by the generic ones.

The interpolation problem \ref{interp} is said to be \emph{extremal}
if it cannot be solved with a holomorphic function $g$ satisfying
$\sup_{\D^d} |g| <1$.  There is no harm in multiplying $\sigma, \tau$
by $r>1$ if necessary to force the problem to be extremal.  

The interpolation problem \ref{interp} is said to be
\emph{non-degenerate} if no two-point subproblem is extremal.  If a
two point subproblem is extremal (the degenerate case) then one of the
following holds
\[
|\sigma| = \max_{j=1,\dots, d} |z_j| \text{ or } 
|\tau| = \max_{j=1,\dots, d} |w_j| \text{ or } 
\rho(\sigma, \tau) = \max_{j=1,\dots, d} \rho(z_j,w_j).
\]
By our genericity assumption, whichever maximum occurs above will
occur at a unique $j$ and this forces the solution function to be
unique and to depend on one variable.  We provide some details in
Lemma \ref{details} at the end of this section.

Thus, in the degenerate case we can certainly solve with a function of
the form \eqref{form} and we may now assume our interpolation problem
is non-degenerate.

To finish, we may quote appropriate results from \cite{LK}.  Lemma 3 of
\cite{LK} states that for $d=3$, if the interpolation problem
\eqref{interp} is extremal, non-degenerate, and strictly
3-dimensional, then it can be solved with a function of the form
\eqref{form}.  Strictly 3-dimensional means the problem can be solved
with a function depending on 3 variables but not with a function only
depending on 2 variables.  We may certainly assume that we are in the
strictly 3-dimensional case since otherwise there is nothing to prove.

For $d>3$, Lemma 5 of \cite{LK} states that if \eqref{interp} is
extremal and non-degenerate, then after permuting variables if
necessary the problem can be solved with a function of the form
\eqref{form}.  This completes our explanation of Kosi\'nski's Theorem
\ref{interp}.

We used the following fact above.
\begin{lemma} \label{details} Suppose $f:\D^d \to \D$ is holomorphic,
  $f(0)=0$ and there exists $w \in \D^d$ such that $f(w) = w_1$ and
  $|w_1| > |w_j|$ for all $j\ne 1$.  Then, $f(z) \equiv z_1$ for all
  $z \in \D^d$.
\end{lemma}

\begin{proof} For $\zeta \in \D$, let $h(\zeta) = f(\frac{\zeta}{w_1}
  w)$.  Then, $h(0)=0, h(w_1) = w_1$ and therefore by the classical
  Schwarz lemma, $h(\zeta) \equiv \zeta$.  This implies
\[
\sum_{j=1}^{d} \frac{\partial f}{\partial z_j}(0) \frac{w_j}{w_1} = 1.
\]
By \cite{Rudin} (page 179), $\sum_{j=1}^{d} \left|\frac{\partial
    f}{\partial z_j}(0)\right| \leq 1$.  Since $|w_1| > |w_j|$ for
$j\ne 1$, this can only happen if $\frac{\partial f}{\partial z_1}(0)
= 1$.  This implies $f(\zeta,0,\dots,0) \equiv \zeta$.  By Lemma 3.2
of \cite{GK}, this implies $f(z) \equiv z_1$.
\end{proof}

\section{The Schur-Agler class}

The Schur-Agler class on the polydisc $\D^d$ consists of holomorphic
functions $f:\D^d \to \cd$ such that 
\[
\|f(T)\| \leq 1
\]
for all $d$-tuples $T$ of commuting strictly contractive operators.
Functions of the form \eqref{form} are certainly in the Schur-Agler
class; indeed the argument after the statement of Theorem \ref{interp}
proves this.  

Thus, three-point Pick interpolation problems on the polydisc can be
solved with functions in the Schur-Agler class.  However, the
Schur-Agler class has a well-known interpolation theorem due to Agler;
see \cite{AM}.  Combining these observations we get the following.

\begin{theorem} Given $z_1,z_2,z_3 \in \D^d$ and $t_1,t_2,t_3 \in \D$,
  there exists a holomorphic function $f:\D^d\to \D$ satisfying
  $f(z_j) = t_j$ for $j=1,2,3$ if and only if there exist positive
  semi-definite $3\times 3$ matrices $\Gamma^1, \Gamma^2,\dots,
  \Gamma^d$ such that
\[
1-t_j \overline{t_k} = \sum_{n=1}^{d} (1-z_j^n \overline{z_k^n})
\Gamma_{j,k}^n
\]
for $j,k=1,2,3$.  Here superscripts are used to denote components of
$z_j \in \D^d$.  
\end{theorem}

This result can be used to prove that every solvable 3-point Pick
problem can be solved with a rational inner function in the
Schur-Agler class.  The paper \cite{LK} already proves this in the
case of non-degenerate extremal problems but then the above general
machinery can be used to show that every solvable problem, whether
extremal or not, has a rational inner solution.

\end{document}